\documentclass[12pt]{amsart}
\usepackage{amssymb}
\usepackage[all]{xy}
\newtheorem{theorem}{Theorem}[section]
\newtheorem{lemma}[theorem]{Lemma}

\newtheorem{definition}[theorem]{Definition}

\newtheorem{remark}[theorem]{Remark}
\newtheorem{corollary}[theorem]{Corollary}
\newtheorem{example}[theorem]{Example}
\def\p{\partial}
\def\to{\rightarrow}
\def\c{\mathcal}
\def\r{\mathrm}

\def\f{\mathfrak}
\def\bb{\mathbb}

\begin{document}
\baselineskip15pt
\title[Stability of first order PDEs]
{On the stability of first order PDEs associated to vector fields}
\author[M.M. Sadr]{Maysam Maysami Sadr}
\address{Department of Mathematics\\
Institute for Advanced Studies in Basic Sciences\\
P.O. Box 45195-1159, Zanjan 45137-66731, Iran}
\email{sadr@iasbs.ac.ir}
\begin{abstract}
Let $\c{M}$ be a manifold, $\c{V}$ be a vector field on $\c{M}$, and $\c{B}$ be a Banach space.
For any fixed function $f:\c{M}\to\c{B}$ and any fixed complex number $\lambda$,
we study Hyers-Ulam stability of the global differential equation $\c{V}y=\lambda y+f$.
\end{abstract}
\maketitle
\section{Introduction}
Since 1941 that Hyers \cite{Hyers1} gave a solution for a famous stability problem proposed by Ulam \cite{Ulam1},
Hyers-Ulam stability problem (HUS for short) have been considered by many authors for various types of functional equations.
For historical developments in this area see \cite{Jung1}.
The following four lines describe briefly the general framework of HUS:
\begin{enumerate}
\item[(i)] Take a type of functional equation;
\item[(ii)] take a notion of `approximate solution'; and
\item[(iii)] take a notion of `distance' between approximate solutions.
\item[(iv)] Then find conditions under which the following HUS rule is satisfied:
Sufficiently close to any approximate solution there exists an exact solution.
\end{enumerate}
Let us explain the framework by a traditional example. Consider the Cauchy functional equation:
$$f(xy)=f(x)+f(y).$$
This equation is meaningful for functions $f:\c{G}\to\c{B}$ where $\c{G}$ is an arbitrary (multiplicative) group and $\c{B}$ is an arbitrary
Banach space. Indeed, in (i) above by `a type of functional equation' we mean a formal equation that can be expressed in a similar form
for functions with various domains and ranges. For $\delta\geq0$, $g:\c{G}\to\c{B}$ is called a $\delta$-approximate solution if
$\sup_{x,y\in\c{G}}\|g(xy)-g(x)-g(y)\|\leq\delta$. The distance between two approximate solutions $g,g'$ is defined by the supremum norm
$\|g-g'\|_\infty:=\sup_{x\in\c{G}}\|g(x)-g'(x)\|$. Then the problem is to characterize those groups $\c{G}$ and Banach spaces $\c{B}$ for which
the following HUS rule is satisfied:
\begin{quote}
For every $\epsilon\geq0$ there exists $\delta\geq0$ such that for every $\delta$-approximate solution $g$
there is an exact solution $f$ with $\|f-g\|_\infty\leq\epsilon$.
\end{quote}
A solution for this specific HUS problem has been given by Forti. He showed in \cite{Forti1} that HUS rule is satisfied for
every amenable group $\c{G}$ and for $\c{B}=\mathbb{R}$, the Banach space of real numbers.
For another example of this kind see \cite{Sadr1}.

Since over a decade ago HUS for differential equations has been considered by many authors.
See \cite{Jung3,Jung5,MiuraMiyajimaTakahasi1,MiuraJungTakahasi1,WangZhouSun1,PopaRasa1,CimpeanPopa1} and references therein.
Almost all papers written in this context have dealt with local ODEs or PDEs.
Here by `local equation' we mean a differential equation on a region in Euclidean space (or a differential equation on a manifold
which depends explicitly to a coordinate system). Our main goal of this note is to bring the attention of researchers to HUS of
global differential equations on smooth manifolds. Here by `a global equation' we mean a differential equation that can be expressed
for functions (or tensors) on a class of manifolds and hence dose not depend on any particular coordinate system. Thus `a global differential equation'
can be considered as `a type of functional equation'.
There are two reasons for investigating HUS for global differential equations:
The first one is that at least from a geometric viewpoint existence of global solutions of differential equations on manifolds are much more important
than existence of local solutions. Also many types of qualitative properties of the solutions are meaningful only in the global cases.
The second reason is that the form of a global differential equation on a fixed manifold changes from a coordinate system to another one.
Thus HUS of the global equation may imply HUS of seemingly different local equations.
For example, our main result (Theorem \ref{T1}) recovers and generalizes
results of \cite{CimpeanPopa1,Jung5} with a unified proof and rather simpler than proofs in the mentioned papers.
See examples and remarks in Section \ref{S2}.

Let us offer a simple framework for investigation of HUS of global differential equations:
\begin{enumerate}
\item[(a)] Take a global linear differential operator $\c{D}$ which acts on Banach valued functions on each manifold in a specific
class $\mathfrak{M}$ of (finite or infinite dimensional) manifolds.
\item[(b)] Consider the differential equation
\begin{equation}\label{E00}
\c{D}y=f
\end{equation}
for functions $y:\c{M}\to\c{B}$ where $\c{M}$ belongs to $\f{M}$, $\c{B}$ is a Banach space,
and $f:\c{M}\to\c{B}$ is a fixed map. Let the HUS rule be expressed as follows: For every $\epsilon\geq0$ there is $\delta\geq0$ such that
if $z:\c{M}\to\c{B}$ is a $\delta$-approximate solution of (\ref{E00}) (i.e. $\|\c{D}z-f\|_\infty\leq\delta$)
then there exists an exact solution $y$ such that $\|y-z\|_\infty\leq\epsilon$.
\item[(c)] Then the HUS problem is as follows: Characterize those triples $(\c{M},\c{B},f)$ for which the HUS rule is satisfied.
 \end{enumerate}
For example, let $\f{M}$ be class of Riemannian manifolds $(\c{M},g)$ and $\c{D}:=\Delta_g$ the Laplace operator.
There are many interesting variants of the above framework. By a little work the framework can be restated for global nonlinear
differential operators $\c{D}$ or for global (pseudo) differential operators $\c{D}$ which act on sections of vector bundles.
Thus, for instance, the HUS of Equation (\ref{E00}) in the case that $\f{M}$ is class of spin manifolds, $\c{D}$ is the Dirac operator,
and $f$ is a spinor field, can be defined.

In this note we consider HUS of one of the most simple global PDEs.
Indeed, following the notations of our framework we let $\f{M}$ be the class of triples $(\c{M},\c{V},\lambda)$ where $\c{M}$ is a manifold,
$\c{V}$ is a vector field on $\c{M}$, and $\lambda$ is a complex number; and we let $\c{D}$ be the first order linear partial differential operator
defined by $\c{D}y:=\c{V}y-\lambda y$. The idea for considering this subject arose from \cite{CimpeanPopa1}.
In \cite{CimpeanPopa1} Cimpean and Popa studied HUS of the (local) differential equation (called Euler's equation),
\begin{equation}\label{E0.5}
x_1\frac{\p y}{\p x_1}+\cdots+x_n\frac{\p y}{\p x_n}=\lambda y
\end{equation}
for functions $y$ from $\bb{R}^n_+$ to a Banach space, where $\lambda$ is a real number. They showed that if $\lambda\neq0$
and if a function $z$ is an approximate solution of (\ref{E0.5}), that is
$$\|x_1\frac{\p z}{\p x_1}(x)+\cdots+x_n\frac{\p z}{\p x_n}(x)-\lambda z(x)\|\leq\epsilon\hspace{10mm}(x\in\bb{R}^n_+),$$
then there is an exact solution $y$ of (\ref{E0.5}) satisfying $\|y-z\|_\infty\leq\epsilon\lambda^{-1}$. In other words, (\ref{E0.5})
has HUS if $\lambda\neq0$. From a geometric viewpoint, (\ref{E0.5}) is of the form $\c{V}y=\lambda y$ where $\c{V}$ denotes the vector field
on $\bb{R}^n_+$ defined by $x\mapsto x\partial_x$. We will see in Example \ref{Ex1} that this result can be generalized to a rather big class of
vector fields on $\bb{R}^n_+$.

Let us restate exactly our HUS framework in this spacial case but this time we also take care about smoothness properties of
involved manifolds and functions. Let $1\leq p\leq\infty$. From now on we let $\c{M}$ denote a $\r{C}^{p+1}$-manifold of dimension
$n\geq1$, $\c{V}$ a $\r{C}^p$-vector field on $\c{M}$, and $\c{B}$ a Banach space. For $k\geq0$, $\r{C}^k(\c{M};\c{B})$ denotes the space
of all $\r{C}^k$ maps from $\c{M}$ into $\c{B}$.
Suppose that $y:\c{M}\to\c{B}$ is differentiable along every integral curve of $\c{V}$. Then we let $\c{V}y:\c{M}\to\c{B}$ be defined
by $\c{V}y(x)=(y\circ\gamma_x)'(0)$ where $\gamma_x$ is an integral curve of $\c{V}$ beginning at $x$.
(Note that the $\r{C}^1$-smoothness property of $\c{V}$ implies that $\c{V}y$ is well-defined, see \cite[Theorem 2.1]{Lang1}.)
We denote by $\r{C}_\c{V}^1(\c{M};\c{B})$ the class of those functions $y:\c{M}\to\c{B}$ such that $y\in\r{C}^0(\c{M};\c{B})$,
$y$ is differentiable along every integral curve of $\c{V}$, and such that $\c{V}y\in\r{C}^0(\c{M};\c{B})$.
Following our framework and also \cite{MiuraJungTakahasi1,MiuraMiyajimaTakahasi1,WangZhouSun1} we make the definition below.
\begin{definition}
Let $\lambda$ be a fixed complex number and $f:\c{M}\to\c{B}$ be a fixed map. Consider the following differential equation for functions
$y:\c{M}\to\c{B}$.
\begin{equation}\label{E1}
\c{V}y=\lambda y+f
\end{equation}
We say that Equation (\ref{E1}) has $\r{C}_\c{V}^1$-HUS if the following statement holds.
There exists a real constant $K\geq0$ such that for every $\epsilon\geq0$ and every $y\in\r{C}_\c{V}^1(\c{M};\c{B})$ if
\begin{equation}\label{E2}
\|\c{V}y-\lambda y-f\|_\infty\leq\epsilon,
\end{equation}
then there is a $z\in\r{C}_\c{V}^1(\c{M};\c{B})$ such that
\begin{equation}\label{E2.5}
\c{V}z=\lambda z+f\hspace{5mm}\text{and}\hspace{5mm}\|y-z\|_\infty\leq K\epsilon.
\end{equation}
We call $K$ a stability constant for Equation (\ref{E1}).
Let $1\leq p_1,p_2\leq p+1$. We say that
(\ref{E1}) has $\r{C}^{p_1\to p_2}$-HUS if the above statement holds while $y\in\r{C}^{p_1}(\c{M};\c{B})$ and $z\in\r{C}^{p_2}(\c{M};\c{B})$.
\end{definition}
Our main result (Theorem \ref{T1}) asserts that under some mild conditions Equation (\ref{E1})
has $\r{C}_\c{V}^1$-HUS. The method of the proof is very simple: (\ref{E1}) becomes a first order ODE along every integral curve of
$\c{V}$; then one can apply the results of  \cite{PopaRasa1}.
\section{The main result}\label{S2}
The following lemma is a special case of \cite[Theorem 2.2]{PopaRasa1}.
\begin{lemma}\label{L1}
Suppose that the real part of $\lambda$ is nonzero: $\f{R}\lambda\neq0$.
Let the inequality
$$\|a'-\lambda a-h\|_\infty\leq\epsilon$$
be satisfied for $a\in\r{C}^1(\bb{R};\c{B})$ and $h\in\r{C}^0(\bb{R};\c{B})$. Then there is a unique $b\in\r{C}^1(\bb{R};\c{B})$ satisfying
$b'=\lambda b+h$ and $\|a-b\|_\infty<\infty$. Moreover, $\|a-b\|_\infty\leq\epsilon|\f{R}\lambda|^{-1}$ and $b$ is explicitly given by
\begin{equation}\label{E2.75}
b(t)=e^{\lambda t}[\int_0^th(s)e^{-\lambda s}\r{d}(s)+\int_0^\omega\alpha(s)e^{-\lambda s}\r{d}(s)+a(0)],
\end{equation}
where $\alpha=a'-\lambda a-h$ and $\omega=+\infty$ if $\f{R}\lambda>0$ and otherwise $\omega=-\infty$.
\end{lemma}
\begin{proof}
Since $a'=\lambda a+h+\alpha$ we have
$a(t)=e^{\lambda t}[\int_0^t(h+\alpha)(s)e^{-\lambda s}\r{d}(s)+a(0)]$. Let $b$ be given by (\ref{E2.75}). Then $b'=\lambda b+h$ and we have
$\|a(t)-b(t)\|=e^{\f{R}\lambda t}\|\int_t^\omega\alpha(s)e^{-\lambda s}\r{d}(s)\|\leq\epsilon|\f{R}\lambda|^{-1}$.
Let $c\in\r{C}^1(\bb{R};\c{B})$ be such that $c'=\lambda c+h$ and $\|a-c\|_\infty<\infty$. Then
$\|c-b\|_\infty<\infty$ and also $c(t)=e^{\lambda t}[\int_0^th(s)e^{-\lambda s}\r{d}(s)+c(0)]$.
It follows from this latter equality and (\ref{E2.75}) that
$\|c(t)-b(t)\|=e^{\f{R}\lambda t}\|c(0)-a(0)-\int_0^\omega\alpha(s)e^{-\lambda s}\r{d}(s)\|$. Thus,
$c(0)=a(0)+\int_0^\omega\alpha(s)e^{-\lambda s}\r{d}(s)$. This completes the proof.
\end{proof}
Our main result is as follows.
\begin{theorem}\label{T1}
Suppose that $\c{V}$ is a complete vector field, $\f{R}\lambda\neq0$, and $f\in\r{C}^0(\c{M};\c{B})$.
Then Equation (\ref{E1}) has $\r{C}_\c{V}^1$-HUS with stability constant $|\f{R}\lambda|^{-1}$.
\end{theorem}
\begin{proof}
By completeness of $\c{V}$, the domain of the flow of $\c{V}$ is the whole $\bb{R}\times\c{M}$,
i.e. there is a smooth action $\Phi:\bb{R}\times\c{M}\to\c{M}$ such that $\Phi(0,x)=x$,
$\Phi(t+s,x)=\Phi(t,\Phi(s,x))$, and $\frac{\p\Phi}{\p t}(t,x)=\c{V}\circ\Phi(t,x)$ for every $x\in\c{M}$ and every $s,t\in\bb{R}$.
Suppose that $y\in\r{C}_\c{V}^1(\c{M};\c{B})$ satisfies (\ref{E2}).
For every $x$, $a_x:=y\circ\Phi(\cdot,x)$ belongs to $\r{C}^1(\bb{R};\c{B})$ and satisfies
$$\|a'_x-\lambda a_x-f\circ\Phi(\cdot,x)\|_\infty\leq\epsilon.$$
By Lemma \ref{L1}, there is a unique $b_x\in\r{C}^1(\bb{R};\c{B})$ such that
\begin{equation}\label{E3}
b_x'=\lambda b_x+f\circ\Phi(\cdot,x)\hspace{5mm}\text{and}\hspace{5mm}\|a_x-b_x\|_\infty\leq\epsilon|\f{R}\lambda|^{-1}.
\end{equation}
Let $z:\c{M}\to\c{B}$ be defined by $z(x):=b_x(0)$. It follows that
\begin{equation}\label{E4}
z(x)=a_x(0)+\int_0^\omega[a'_x(s)-\lambda a_x(s)-f\circ\Phi(s,x)]e^{-\lambda s}\r{d}(s),
\end{equation}
where $\omega$ is as in Lemma \ref{L1}. Then a simple application of Lebesgue's Dominated Convergence Theorem shows that $z$ is continuous.
Let $t$ and $x$ be arbitrary and for a while be fixed.
Define a function $c:\bb{R}\to\c{B}$ by $c(s):=b_x(s+t)$. Then by (\ref{E3}) we have $c'(s)=\lambda c(s)+f\circ\Phi(s,\Phi(t,x))$
and $\|a_{\Phi(t,x)}(s)-c(s)\|\leq\epsilon|\f{R}\lambda|^{-1}$. Thus the uniqueness property mentioned in Lemma \ref{L1} implies that
$c=b_{\Phi(t,x)}$. Therefore we have proved that,
\begin{equation}\label{E7}
z(\Phi(t,x))=b_x(t).
\end{equation}
This identity together with (\ref{E3}) shows that
$z$ is differentiable along every integral curve of $\c{V}$ and satisfies (\ref{E2.5}) with $K=|\f{R}\lambda|^{-1}$. Also
it follows from the identity $\c{V}z=\lambda z+f$ that $z\in\r{C}_\c{V}^1(\c{M};\c{B})$. This completes the proof.
\end{proof}
The following corollary of Theorem \ref{T1} gives a bound for approximate solutions of Equation (\ref{E1}) in the case that $\c{V}$ is
periodic and $f=0$. This result can be considered as an application of HUS.
\begin{corollary}
Suppose that $\f{R}\lambda\neq0$ and $\c{V}$ is a complete vector field for which every integral curve is periodic
(constant integral curves are included).
Then for every $y\in\r{C}_\c{V}^1(\c{M};\c{B})$ satisfying $\|\c{V}y-\lambda y\|_\infty\leq\epsilon$, the inequality
$\|y\|_\infty\leq\epsilon|\f{R}\lambda|^{-1}$ holds.
\end{corollary}
\begin{proof}
By Theorem \ref{T1}, there is a $z\in\r{C}_\c{V}^1(\c{M};\c{B})$ such that $\c{V}z=\lambda z$ and $\|y-z\|_\infty\leq\epsilon|\f{R}\lambda|^{-1}$.
For every $x\in\c{M}$ let $b_x$ be as in the proof of Theorem \ref{T1} with $f=0$. It follows from (\ref{E7}) that $b_x$ is a periodic function.
On the other hand, $b_x'=\lambda b_x$. So, $b_x=0$ and hence $z=0$. This completes the proof.
\end{proof}
The following HUS result follows from Theorem \ref{T1}.
\begin{theorem}
Let $2\leq p\leq\infty$. Suppose that $\c{M}$ is a closed manifold, $\f{R}\lambda\neq0$, and $f\in\r{C}^{p-1}(\c{M};\c{B})$.
Then Equation (\ref{E1}) has $\r{C}^{p\to p-1}$-HUS with stability constant $|\f{R}\lambda|^{-1}$.
\end{theorem}
\begin{proof}
Suppose that $y\in\r{C}^p(\c{M};\c{B})$ satisfies (\ref{E2}). Since
$\c{M}$ is closed, $\c{V}$ is automatically complete. By Theorem \ref{T1} there is a $z\in\r{C}_\c{V}^1(\c{M};\c{B})$ satisfying (\ref{E2.5})
with $K=|\f{R}\lambda|^{-1}$. Also it follows from (\ref{E4}) and boundedness of higher derivatives of $y$ and $f$
that $z\in\r{C}^{p-1}(\c{M};\c{B})$. This completes the proof.
\end{proof}
\begin{example}\label{Ex1}
\emph{Let $\c{M}$ be an open subset of $\bb{R}^n$ such that if $t>0$ and $x\in\c{M}$
then $tx\in\c{M}$. Suppose that $g:\c{M}\to\bb{R}$ is a $\r{C}^1$-function which is homogenous of degree zero, i.e. $g(tx)=g(x)$
for every $x\in\c{M}$, and $t>0$. Define a $\r{C}^1$-vector field $\c{V}:\c{M}\to\bb{R}^n$ by $x\mapsto g(x)x$.
Then for every $x\in\c{M}$ the integral curve of $\c{V}$ beginning at $x$ is given by
$t\mapsto e^{tg(x)}x$. Thus $\c{V}$ is complete and if $\f{R}\lambda\neq0$ and $f\in\r{C}^0(\c{M},\c{B})$
then it follows from Theorem \ref{T1} that,
$$\sum_{i=1}^n x_ig\frac{\p y}{\p x_i}=\lambda y+f,$$
has $\r{C}_\c{V}^1$-HUS. This generalizes \cite[Theorem 2.2(1)]{CimpeanPopa1}.}
\end{example}
\begin{example}
\emph{Let $\c{M}=\bb{R}^n$, $M=(m_{ij})$ be a real $n\times n$ matrix, and $v=(v_1,\ldots,v_n)$ be a vector in $\bb{R}^n$ such that $Mv=0$.
Let the vector field $\c{V}$ be defined by $x\mapsto Mx+v$. Then for every $x$ the curve defined by $t\mapsto e^{tM}x+tv$ is the integral
curve of $\c{V}$ beginning at $x$. Thus $\c{V}$ is complete and if $\f{R}\lambda\neq0$ and $f\in\r{C}^0(\c{M},\c{B})$
then it follows from Theorem \ref{T1} that,
$$\sum_{i=1}^n(v_i+\sum_{j=1}^nm_{ij}x_j)\frac{\p y}{\p x_i}=\lambda y+f$$
has $\r{C}_\c{V}^1$-HUS. This generalizes some of the results of \cite{Jung5}.}
\end{example}
We list three classes of very well-known examples of complete vector fields:
Every vector field with compact support is complete. Left or right invariant vector fields on Lie groups are complete.
The geodesic vector field on the tangent bundle of a complete Riemannian manifold is complete. (Let us write explicitly this latter vector field
in a local coordinate system: Let $X$ be a Riemannian manifold of dimension $k$, let $(x_1,\ldots,x_{k})$
denote a point in a local coordinate system of $X$ and let $(x_{1},\ldots,x_{2k})$ denote a point in the corresponding local coordinate
system of the tangent bundle of $X$. Then the geodesic vector field is given by
$\sum_{i=1}^kx_{k+i}\frac{\p }{\p x_i}-\sum_{i,j,j'=1}^{k}\Gamma_{j,j'}^ix_{k+j}x_{k+j'}\frac{\p }{\p x_{k+i}}$
where $\Gamma_{j,j'}^i$ denotes the Christoffel symbol of the Riemannian connection in the local coordinate system.)

We end this paper by some remarks about Equation (\ref{E1}) and its stability.
\begin{remark}
\emph{\begin{enumerate}
\item[(i)] Results \cite[Theorem 2.8]{PopaRasa1} and \cite[Theorem 2.2(2)]{CimpeanPopa1} show that the assumption $\f{R}\lambda\neq0$
can not be removed from the statement of Theorem \ref{T1}.
\item[(ii)] Let $\c{M}$ be a Riemannian manifold. The basic properties of gradient vector fields implies that the
following three identities are equivalent for every $f,y\in\r{C}^2(\c{M};\bb{R})$.
$$(\r{grad}f)y=y+f,\hspace{3mm}(\r{grad}y)f=f+y,\hspace{3mm}\langle\r{grad}f,\r{grad}y\rangle=f+y.$$
\item[(iii)] The Hamiltonian version of the above (Lagrangian) statement is as follows. Let $\c{M}$ be a symplectic manifold.
The following three identities are equivalent for every $f,y\in\r{C}^2(\c{M};\bb{R})$.
($\{,\}$ denotes the corresponding Poisson bracket and $\chi_f$ denotes the Hamiltonian vector field associated to $f$.)
$$\chi_fy=y+f,\hspace{3mm}\chi_yf=-(f+y),\hspace{3mm}\{f,y\}=f+y.$$
\end{enumerate}}
\end{remark}
\bibliographystyle{amsplain}

\end{document}